\documentclass[10pt]{amsart}
\usepackage{fancyhdr}
\usepackage{stmaryrd}
\usepackage{calrsfs}

\usepackage{amsmath}
\usepackage[nospace,noadjust]{cite}
\usepackage{amsfonts}
\usepackage{amssymb,enumerate}
\usepackage{amsthm}
\usepackage{cite}
\usepackage{comment}
\usepackage{color}
\usepackage[all]{xy}
\usepackage{hyperref}
\usepackage{lineno}
\usepackage{tikz-cd}

\newtheorem*{maintheorem*}{Main Theorem}
\newtheorem{theorem}{Theorem}[section]
\newtheorem{prop}[theorem]{Proposition}
\newtheorem{question}[theorem]{Question}

\theoremstyle{definition}
\newtheorem{definition}[theorem]{Definition}
\newtheorem{remark}[theorem]{Remark}
\newtheorem{example}[theorem]{Example}
\numberwithin{equation}{section}

\newcommand{\nn}{\mathbb{N}}
\newcommand{\pp}{\mathbb{P}}
\newcommand{\qq}{\mathbb{Q}}
\newcommand{\rr}{\mathbb{R}}
\newcommand{\zz}{\mathbb{Z}}

\newcommand{\gp}{\text{gp}}

\newcommand{\uu}{\mathcal{U}}

\providecommand\ldb{\llbracket}
\providecommand\rdb{\rrbracket}

\hypersetup{
	pdftitle={FD},
	colorlinks=true,
	linkcolor=blue,
	citecolor=cyan,
	urlcolor=wine
}



\begin{document}
	
	\mbox{}
	\title{On the atomic structure of torsion-free monoids} 
	
	\author{Felix Gotti}
	\address{Department of Mathematics\\MIT\\Cambridge, MA 02139}
	\email{fgotti@mit.edu}
	
	\author{Joseph Vulakh}
	\address{Paul Laurence Dunbar School, Lexington, KY 40513}
	\email{joseph@vulakh.us}
	
\date{\today}

\begin{abstract}
	 Let $M$ be a cancellative and commutative (additive) monoid. The monoid $M$ is atomic if every non-invertible element can be written as a sum of irreducible elements, which are also called atoms. Also, $M$ satisfies the ascending chain condition on principal ideals (ACCP) if every increasing sequence of principal ideals (under inclusion) becomes constant from one point on. In the first part of this paper, we characterize torsion-free monoids that satisfy the ACCP as those torsion-free monoids whose submonoids are all atomic. A submonoid of the nonnegative cone of a totally ordered abelian group is often called a positive monoid. Every positive monoid is clearly torsion-free. In the second part of this paper, we study the atomic structure of certain classes of positive monoids.
\end{abstract}
\medskip

\maketitle


\bigskip
\section{Introduction}
\label{sec:intro}
\smallskip

A cancellative and commutative additive monoid is called atomic if every non-invertible element can be expressed as a sum of irreducibles, also called atoms. Motivated by the celebrated paper~\cite{AAZ90} by Anderson, Anderson and Zafrullah, the property of being atomic has received a great deal of attention in the literature during the last three decades. Although in~\cite{AAZ90} the authors only considered atomicity in the context of integral domains, the same notion has been significantly extended to and explored in a variety of different contexts. Perhaps the first influential generalization is due to Halter-Koch, who generalized in~\cite{fHK92} to the more general context of cancellative and commutative monoids some atomic notions introduced in~\cite{AAZ90} for integral domains. Although it goes outside the scope of this paper, it is worth emphasizing that atomicity has also been studied in non-cancellative and non-commutative algebraic structures, including commutative rings with nonzero zero-divisors~\cite{EJ21,JM22}, non-commutative monoids \cite{BG20,BS15}, and even more general scenarios~\cite{CT22,CT23}.
\smallskip

Even before the nineties, atomicity was sporadically studied in the context of integral domains, mainly in connection to the ACCP. A cancellative and commutative monoid (or an integral domain) satisfies the ascending chain condition on principal ideals (ACCP) if every increasing sequence of principal ideals (under inclusion) becomes constant from one point on. The first example of an atomic domain not satisfying the ACCP was constructed by Grams in~\cite{aG74}, correcting Cohn's misled assertion that being atomic and satisfying the ACCP were equivalent conditions in the context of integral domains. Later in the eighties, Zaks constructed in~\cite{aZ82} two more examples of atomic domains that do not satisfy the ACCP (one of them was suggested by Cohn). Further examples of atomic domains that do not satisfy the ACCP have been constructed since then (see \cite{mR93,BC19,GL22,GL23}), although none of these examples has a trivial construction.
\smallskip

A cancellative and commutative monoid $M$ is called hereditarily atomic provided that every submonoid of $M$ is atomic. Hereditary atomicity in the context of integral domains was recently studied by Coykendall, Hasenauer, and the first author in~\cite{CGH23}. In this paper, we use hereditary atomicity to characterize the ACCP property in the context of reduced torsion-free monoids (a monoid is called reduced if its group of invertible elements is trivial, while it is called torsion-free if its difference group is torsion-free). Indeed, the fundamental result we establish in this paper, Theorem~\ref{thm:ACCP=HA}, states that a reduced torsion-free monoid satisfies the ACCP if and only if it is hereditarily atomic. It is unknown to the authors whether these two conditions are equivalent after dropping the torsion-free condition. To motivate further research we leave this as an open question at the end of Section~\ref{sec:ACCP and HA}.
\smallskip

A positive monoid is a submonoid of the nonnegative cone of a totally ordered abelian group. Positive monoids form a special class of torsion-free monoids. In Section~\ref{sec:positive monoids}, we explore various atomic aspects of certain classes of positive monoids. We consider properties stronger than atomicity, including the bounded and the finite factorization properties. Both of these properties were introduced in~\cite{AAZ90} to better understand the atomicity of integral domains in the context of the methodological diagram illustrated in Figure~\ref{fig:AAZ chain fragment},
\begin{figure}[h]
	\begin{tikzcd}
		\textbf{ UFM } \ \arrow[r, Rightarrow] & \ \textbf{ FFM } \arrow[r, Rightarrow] & \ \textbf{ BFM } \arrow[r, Rightarrow] & \ \textbf{ ACCP }
	\end{tikzcd}
	\caption{A monoid adaptation of a fragment of the diagram introduced in 1990 by Anderson, Anderson, and Zafrullah to study factorizations in the context of integral domains.}
	\label{fig:AAZ chain fragment}
\end{figure}
where UFM (resp., FFM and BFM) stands for unique factorization monoid (resp., finite factorization monoid and bounded factorization monoid). We also consider properties weaker than atomicity, including the property of being almost atomic and that of being quasi-atomic, both introduced by Boynton and Coykendall in~\cite{BC15} to study divisibility in integral domains, as well as the property of being nearly atomic, recently introduced by Lebowitz-Lockard in \cite{nLL19}. These properties also form a nested diagram of atomic classes: this is illustrated in Figure~\ref{fig:subatomic properties}, where the non-standard acronym ATM (resp., NAM, AAM, QAM) stands for atomic monoid (resp., nearly atomic monoid, almost atomic monoid, quasi-atomic monoid).
\begin{figure}[h]
	\begin{tikzcd}
		\textbf{ ATM } \ \arrow[r, Rightarrow] & \ \textbf{ NAM } \arrow[r, Rightarrow] & \ \textbf{ AAM } \arrow[r, Rightarrow] & \ \textbf{ QAM }
	\end{tikzcd}
	\caption{A chain of implications extending to the right that in Figure~\ref{fig:AAZ chain fragment} and consisting of properties weaker than being atomic.}
	\label{fig:subatomic properties}
\end{figure}
Our primary purpose in Section~\ref{sec:positive monoids} is to consider the atomic structure of certain classes of positive monoids that are often useful to construct needed (counter)examples in factorization theory.

\bigskip
\section{Background}
\label{sec:background}

Following usual conventions, we let $\zz$, $\qq$, and $\rr$ denote the set of integers, rational numbers, and real numbers, respectively. We let $\nn$ and $\nn_0$ denote the set of positive and nonnegative integers, respectively. In addition, we let $\pp$ denote the set of primes. For $b,c \in \zz$ with $b \le c$, we let $\ldb b, c \rdb$ denote the set of integers between $b$ and $c$; that is, $\ldb b,c \rdb := \{m \in \zz \mid b \le m \le c\}$. In addition, for $S \subseteq \rr$ and $r \in \rr$, we set $S_{\ge r} = \{s \in S \mid s \ge r\}$ and $S_{> r} = \{s \in S \mid s > r\}$. For $q \in \qq_{>0}$,  the relatively prime positive integers $n$ and $d$ satisfying that $q = \frac nd$ are denoted here by $\mathsf{n}(q)$ and $\mathsf{d}(q)$, respectively.

\smallskip
\subsection{Atomic Notions in Monoids}
\smallskip

A \emph{monoid} is a semigroup with an identity element. However, in the context of this paper we will tacitly assume that monoids are both cancellative and commutative. Let $M$ be a monoid written additively. We set $M^\bullet = M \setminus \{0\}$, and we say that $M$ is \emph{trivial} if $M^\bullet$ is empty. The invertible elements of $M$ form a group, which we denote by $\uu(M)$, and $M$ is called \emph{reduced} if $\uu(M)$ is the trivial group. The \emph{difference group} $\gp(M)$ of $M$ is the unique abelian group $\gp(M)$ up to isomorphism satisfying that any abelian group containing a homomorphic image of $M$ also contains a homomorphic image of $\gp(M)$. The \emph{rank} of $M$ is, by definition, the rank of $\gp(M)$ as a $\zz$-module or, equivalently, the dimension of the vector space $\qq \otimes_\zz \gp(M)$. The rank of $M$ is denoted by $\text{rank} \, M$. The \emph{reduced monoid} of $M$ is the quotient $M/\uu(M)$, which is denoted by $M_{\text{red}}$. For $b,c \in M$, we say that $c$ \emph{divides} $b$ \emph{in} $M$ if there exists $d \in M$ such that $b = c + d$; in this case, we write $c \mid_M b$. A submonoid~$M'$ of $M$ is called a \emph{divisor-closed submonoid} if every element of~$M$ dividing an element of~$M'$ in $M$ belongs to $M'$. If $S$ is a subset of~$M$, then we let $\langle S \rangle$ denote the smallest submonoid of $M$ containing $S$, in which case, we say that $S$ is a \emph{generating set} of $\langle S \rangle$. The monoid $M$ is called \emph{finitely generated} provided that $M = \langle S \rangle$ for some finite subset $S$ of $M$. 
\smallskip

An non-invertible element $a \in M$ is called an \emph{atom} of $M$ if whenever $a = b+c$ for some $b,c \in M$, either $b \in \uu(M)$ or $c \in \uu(M)$. We let $\mathcal{A}(M)$ denote the set consisting of all the atoms of $M$. If $\mathcal{A}(M)$ is empty, $M$ is called \emph{antimatter}. Observe that if $M$ is a reduced monoid, then $\mathcal{A}(M)$ is contained in every generating set of~$M$. An element of~$M$ is called \emph{atomic} if it is invertible or it can be written as a sum of finitely many atoms. The monoid $M$ is called \emph{atomic} if every element of $M$ is atomic. Following~\cite{AAZ90}, we say that $M$ is \emph{strongly atomic} if for any elements $b,c \in M$, there exists an atomic element $d \in M$ that is a common divisor of $b$ and $c$ such that the only common divisors of $b-d$ and $c-d$ are invertible elements. It follows from the definitions that every strongly atomic monoid is atomic, and it is well known that every monoid satisfying the ACCP is strongly atomic (see \cite[Theorem~1.3]{AAZ90}).
\smallskip

We proceed to introduce some properties that are weaker than that of being atomic. An element $c \in M$ is called \emph{quasi-atomic} (resp., \emph{almost atomic}) provided that there exists an element (resp., an atomic element) $b \in M$ such that $b + c$ is atomic. Following~\cite{BC15}, we say that $M$ is \emph{quasi-atomic} (resp., \emph{almost atomic}) if every non-invertible element of $M$ is quasi-atomic (resp., almost atomic). It follows directly from the definitions that every almost atomic monoid is quasi-atomic. Following~\cite{nLL19}, we say that $M$ is \emph{nearly atomic} if there exists $b \in M$ such that for each non-invertible element $c \in M$, the element $b + c$ is atomic. It follows directly from the definitions that every atomic monoid is nearly atomic, and one can prove mimicking the proof of \cite[Lemma~5]{nLL19} that every nearly atomic monoid is almost atomic. Therefore the properties of being nearly atomic, almost atomic, and quasi-atomic are nested weaker notions of atomicity.
\smallskip 

A subset $I$ of $M$ is called an \emph{ideal} of~$M$ if the set $I + M := \{b+c \mid b \in I \text{ and } c \in M\} \subseteq I$ or, equivalently, if $I + M = I$. If~$I$ is an ideal of~$M$ such that $I = b + M := \{b + c \mid c \in M\}$ for some $b \in M$, then $I$ is called \emph{principal}. The monoid $M$ satisfies the \emph{ascending chain condition on principal ideals} (ACCP) provided that every ascending chain $(b_n + M)_{n \in \nn}$ of principal ideals of $M$ stabilizes; that is, there exists $n_0 \in \nn$ such that $b_n + M = b_{n+1} + M$ for every $n \ge n_0$. It is not hard to check that if $M$ satisfies the ACCP, then $M$ is atomic (see \cite[Proposition~1.1.4]{GH06}). As we have mentioned in the introduction, the converse of this statement does not hold.
\smallskip

The free (commutative) monoid on $\mathcal{A}(M_{\text{red}})$ is denoted by $\mathsf{Z}(M)$. Let $\pi \colon \mathsf{Z}(M) \to M_\text{red}$ be the unique monoid homomorphism fixing every element of $\mathcal{A}(M_{\text{red}})$. If $z := a_1 \cdots a_\ell \in \mathsf{Z}(M)$ for some $a_1, \dots, a_\ell \in \mathcal{A}(M_{\text{red}})$, then $\ell$ is called the \emph{length} of $z$ and is denoted by $|z|$. For every $b \in M$, we set
\[
	\mathsf{Z}(b) := \mathsf{Z}_M(b) := \pi^{-1} (b + \uu(M)),
\]
and the elements of $\mathsf{Z}(b)$ are called \emph{factorizations} of~$b$. A recent survey on factorization theory in commutative monoids by Geroldinger and Zhong can be found in~\cite{GZ20}. If $|\mathsf{Z}(b)| = 1$ for every $b \in M$, then~$M$ is called a \emph{unique factorization monoid} (UFM). On the other hand, if $M$ is atomic and $|\mathsf{Z}(b)| < \infty$ for every $b \in M$, then~$M$ is called a \emph{finite factorization monoid} (FFM). It follows directly from the definitions that every UFM is an FFM. In addition, it follows from \cite[Proposition~2.7.8]{GH06} that every finitely generated monoid is an FFM. Now, for every $b \in M$, we set
\[
	\mathsf{L}(b) := \mathsf{L}_M(b) := \{ |z| \mid z \in \mathsf{Z}(b) \}.
\]
If $M$ is atomic and $|\mathsf{L}(b)| < \infty$ for every $b \in M$, then $M$ is called a \emph{bounded factorization monoid} (BFM). According to \cite[Theorem~1]{fHK92}, the monoid $M$ is a BFM if it admits a \emph{length function}, that is, a function $\ell \colon M \to \nn_0$ satisfying the following two properties:
\begin{enumerate}
	\item[(i)] $\ell(u) = 0$ if and only if $u \in \mathcal{U}(M)$;
	\smallskip
	
	\item[(ii)] $\ell(b+c) \ge \ell(b) + \ell(c)$ for all $b,c \in M$.
\end{enumerate}
It follows from the definitions that every FFM is a BFM. Also, it is not hard to argue that every BFM must satisfy the ACCP (see \cite[Corollary~1.4.4]{GH06}). A recent survey on the bounded factorization and finite factorization properties by Anderson and the first author can be found in~\cite{AG22}. The monoid $M$ is called a \emph{length-factorial monoid} (LFM) provided that $M$ is atomic and also that any two distinct factorizations of the same element have different lengths. The notion of length-factoriality was introduced by Coykendall and Smith in~\cite{CS11} under the term `other-half-factoriality'. The same notion has been considered recently by Chapman et al. in~\cite{CCGS21} and by Geroldinger and Zhong in~\cite{GZ21}. It follows directly from the definitions that every UFM is an LFM, and it follows from \cite[Proposition~3.1]{BVZ22} that every LFM is an FFM.
\smallskip

The classes of monoids we have introduced in this subsection are represented in the chain of implications illustrated in Figure~\ref{fig:full atomicity diagram}, which consists of nested classes of monoids and extends simultaneously the diagrams previously shown in Figure~\ref{fig:AAZ chain fragment} and Figure~\ref{fig:subatomic properties}. For the sake of consistency, in the diagram shown in Figure~\ref{fig:full atomicity diagram}, we let the (nonstandard) acronym SAM (resp., ATM, NAM, AAM, and QAM) stand for strongly atomic monoid (resp., atomic monoid, nearly atomic monoid, almost atomic monoid, and quasi-atomic monoid).
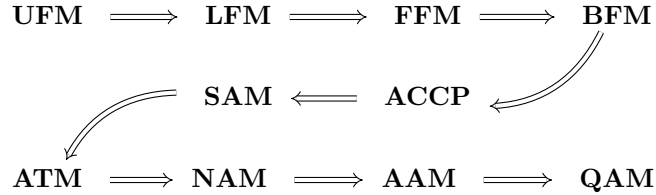
\begin{figure}[h]
	\begin{tikzcd}[cramped]
		\textbf{ UFM } \ \arrow[r, Rightarrow] & \ \textbf{ LFM } \arrow[r, Rightarrow] & \ \textbf{ FFM } \arrow[r, Rightarrow] & \ \textbf{ BFM } \arrow[ld, Rightarrow, bend left] \\ 
		&\ \textbf{ SAM }  \arrow[ld, Rightarrow, bend right] &\ \textbf{ ACCP } \arrow[l, Rightarrow] & \\
		\textbf{ ATM } \ \arrow[r, Rightarrow] & \textbf{ NAM } \ \arrow[r, Rightarrow]	& \textbf{ AAM } \ \arrow[r, Rightarrow] & \ \textbf{ QAM }
	\end{tikzcd}
	\caption{A chain of implications connecting those in Figure~\ref{fig:AAZ chain fragment} and Figure~\ref{fig:subatomic properties} through the property of being strongly atomic.}
	\label{fig:full atomicity diagram}
\end{figure}

\medskip
\subsection{Totally Ordered Groups}
\label{subsec:ordered fields}

We proceed to review some basic notions related to totally ordered abelian groups that we will be using later. A \emph{totally ordered abelian group} is a pair $(G, G^+)$, where $G$ is an abelian (additive) group and $G^+$ is a subset of $G$ containing~$0$ and satisfying the following two conditions:
\begin{enumerate}
	\item for all $g,h \in G^+$,  $g+h \in G^+$, and
	\smallskip
	
	\item for each $g \in G \! \setminus \! \{0\}$, exactly one of the conditions $g \in G^+$ and $-g \in G^+$ holds.
\end{enumerate}
In this case, $G^+$ is called a \emph{positive cone} of $G$, and $G^+$ induces a total order on $G$, namely, $g \le h$ in~$G$ whenever $h-g \in G^+$. On the other hand, if $G$ is an abelian group and $\le$ is a total order on $G$ compatible with the operation of $G$, then $\{g \in G \mid g \ge 0\}$ is a positive cone of $G$ inducing the order~$\le$. We can see that every totally ordered group is torsion-free. When there seems to be no risk of ambiguity, we will often abuse notation, writing $G$ instead of the more cumbersome notation $(G, G^+)$.
\smallskip

Let $G$ be a totally ordered abelian group. For each $g \in G$ set $|g| = g$ if $g \in G^+$ and $|g| = -g$ otherwise. For $g,h \in G$, we write $g = \textbf{O}(h)$ if $|g| \le n|h|$ for some $n \in \nn$, and $g \sim h$ if both $g = \textbf{O}(h)$ and $h = \textbf{O}(g)$ hold. Also, for any $g,h \in G$ with $g = \textbf{O}(h)$ and $g \not\sim h$, we write $g \ll h$ or, equivalently, $h \gg g$. It is clear that $\sim$ defines an equivalence relation on $G$. Let
\[
	v \colon G \setminus \{0\} \to \Gamma_G := (G \setminus \{0\})/\!\sim
\]
be the quotient map. Setting $v(g) \preceq v(h)$ for any $g,h \in G \setminus \{0\}$ with $h = \textbf{O}(g)$, one finds that $(\Gamma_G, \preceq)$ is a well-defined totally ordered set, which is called the \emph{value set} of $G$. 
One can verify that for all $g,h \in G$ such that $0 \notin \{g,h, g+h\}$,
\[
	v(g+h) \ge \min\{v(g), v(h)\},
\]
where the equality holds provided that $v(g) \neq v(h)$ or $g,h \in G^+$. The elements of $\Gamma_G$ are called \emph{Archimedean classes} of $G$, and the quotient map $v$ is called the \emph{Archimedean valuation} on $G$. The totally ordered group $G$ is called \emph{Archimedean} if and only if its value set $\Gamma_G$ is a singleton.
\smallskip

A \emph{positive monoid of a totally ordered group} $G$ is a submonoid of $G$ contained in $G^+$. A monoid~$M$ is called a \emph{positive monoid} if it is isomorphic to a positive monoid of a totally ordered group $G$. According to H\"older's theorem, a totally ordered abelian group is Archimedean if and only if it is order-isomorphic to a subgroup of the additive group $\rr$. Thus, additive submonoids of $\rr_{\ge 0}$ account, up to isomorphism, for all positive monoids of Archimedean groups. In addition, nontrivial additive submonoids of $\qq_{\ge 0}$, also known as \emph{Puiseux monoids}, account up to isomorphism for all the rank-one positive monoids (see \cite[Theorem~3.12]{GGT21} and \cite[Section~24]{lF70}). Cofinite submonoids of the additive monoid $\nn_0$ are called \emph{numerical monoids}. Every numerical monoid is then a Puiseux monoid, and it is well known that a Puiseux monoid is finitely generated if and only if it is isomorphic to a numerical monoid. The atomic structure and arithmetic of factorizations of Puiseux monoids have been systematically studied during the last half-decade (see~\cite{GGT21} and references therein).

\bigskip
\section{Hereditary Atomicity and the ACCP}
\label{sec:ACCP and HA}

It turns out that in the class consisting of reduced torsion-free monoids, being hereditarily atomic and satisfying the ACCP are equivalent conditions.

\begin{theorem} \label{thm:ACCP=HA}
	For a reduced torsion-free monoid $M$, the following conditions are equivalent.
	\begin{enumerate}
		\item[(a)] $M$ satisfies the ACCP.
		\smallskip
		
		\item[(b)] $M$ is hereditarily atomic.
	\end{enumerate}
\end{theorem}
	
\begin{proof}
	(a) $\Rightarrow$ (b): Since $M$ is a reduced monoid, the fact that $M$ satisfies the ACCP immediately implies that every submonoid of $M$ satisfies the ACCP. Now the implication follows from the fact that every monoid satisfying the ACCP is atomic.
	\smallskip
	
	(b) $\Rightarrow$ (a): Let $G$ be the difference group of $M$. Since $M$ is torsion-free, $G$ is a torsion-free abelian group. Then it follows from \cite[Theorem~3.2]{rG84} that $G$ can be turned into a totally ordered group $(G, \le)$ in such a way that $M$ is a submonoid of the nonnegative cone $G^+$ of $G$. Because $G$ is totally ordered, it follows from the Hahn's embedding theorem that $G$ can be embedded as an ordered group into the totally ordered group $\mathfrak{F}(\Gamma_G, \rr)$ consisting of all functions from~$\Gamma_G$ to $\rr$ that vanish outside a well-ordered set, where $\Gamma_G$ is the value set of~$G$. In particular, $G^+$ is a submonoid of $\mathfrak{F}(\Gamma_G, \rr)^+$. Since $\mathfrak{F}(\Gamma_G, \rr)$ is a divisible group, after replacing $G$ by $\mathfrak{F}(\Gamma_G, \rr)$, we can assume that $G$ is a totally ordered divisible abelian group and, in particular, $g$ and $g/n$ are in the same Archimedean class of $G$ for all $g \in G$ and $n \in \nn$.
	
	Suppose, by way of contradiction, that $M$ does not satisfy the ACCP. Then there is a sequence of principal ideals $(q_n + M)_{n \in \nn_0}$ satisfying that $a_{n+1} := q_n - q_{n+1} \in M^\bullet$ for every $n \in \nn_0$. We split the rest of the proof into two cases.
	\smallskip
	
	\textsc{Case 1:} There exists an Archimedean class of $G$ containing infinitely many terms of the sequence $(a_n)_{n \in \nn}$. In this case, there is a subsequence $(b_n)_{n \in \nn}$ of $(a_n)_{n \in \nn}$ all whose terms are in the same Archimedean class of $G$. Now set $q'_n := q_0 - \sum_{i=1}^n b_i$ for each $n \in \nn$. The fact that $(q_n + M)_{n \in \nn}$ is a non-stabilizing ascending chain of principal ideals of $M$ implies that $(q'_n + M)_{n \in \nn}$ is also a non-stabilizing ascending chain of principal ideals of $M$. Therefore, after replacing $(q_n + M)_{n \in \nn}$ by $(q'_n + M)_{n \in \nn}$ if necessary, we can assume that all the terms of the sequence $(a_n)_{n \in \nn}$ belong to the same Archimedean class of $G$. Because $M$ is hereditarily atomic, its submonoid $N := \langle a_n, q_n \mid n \in \nn \rangle$ is atomic. Now for each $n \in \nn_0$ the equality $q_n = q_{n+1} + a_{n+1}$ ensures that $q_n \notin \mathcal{A}(N)$. This, along with the fact that $q_0 = q_1 + a_1 \in N$, implies that  $q_0 \in \langle a_n \mid n \in \nn \rangle$. Hence $a_n \sim q_0$ for every $n \in \nn$.
	
	Now observe that, for any $k \in \nn$, the term $a_k$ cannot be a lower bound of the set $\{a_n \mid n > k\}$ as, otherwise, we could take $n \in \nn$ such that $n a_k > q_0$ to obtain that $\sum_{i=1}^{k+n} a_i > \sum_{i = k+1}^{k+n} a_i \ge n a_k > q_0$, which is not possible. Therefore, after replacing $(a_n)_{n \in \nn}$ by a suitable subsequence and redefining $(q_n)_{n \in \nn}$ as we did in the previous paragraph, we can assume that $(a_n)_{n \in \nn}$ is a strictly decreasing sequence. In the same vein, we observe that for any $k \in \nn$, there must exist $\ell \in \nn_{> k}$ such that $a_k - a_\ell \sim q_0$ as otherwise (i.e., $a_k - a_{k+i} \ll q_0$ for every $i \in \nn$), we could take $n \in \nn$ such that $q_0 < n a_k$ to obtain that $q_0 - \sum_{i=k}^{k+n} a_i < -a_k + \sum_{i=k+1}^{k+n} (a_k - a_i) < 0$, which is not possible. Thus, after replacing $(a_n)_{n \in \nn}$ by a suitable subsequence and redefining $(q_n)_{n \in \nn}$ accordingly, we can assume that $a_n - a_{n+1} \sim q_0$ for every $n \in \nn$. Set $s_n := \sum_{i=1}^n a_i$ for each $n \in \nn$.
	
	Our current goal is to construct a sequence $(a'_n)_{n \in \nn}$ with its terms in $M$ such that, for every $n \in \nn$, the following conditions hold:
	\begin{enumerate}
		\item $q_0 \notin \langle a'_1, \dots, a'_n \rangle$,
		\smallskip
		
		\item $q_0 \sim a'_i$ for every $i \in \ldb 1,n \rdb$, and
		\smallskip
		
		\item $s'_n := \sum_{i=1}^n a'_i$ divides $s_m$ in $M$ for some $m \in \nn$.
	\end{enumerate}
	We proceed inductively. Since $q_0 = \textbf{O}(a_1)$, the set $G^+ \cap \big\{ \frac{q_0}n - a_1 \mid n \in \nn \big\}$ is finite. Now the fact that the sequence $(a_n)_{n \in \nn}$ is strictly decreasing ensures the existence of $i \in \nn$ such that $a_i \notin \big\{ \frac{q_0}n - a_1 \mid n \in \nn \big\}$, whence $q_0 \notin \langle a_1 + a_i \rangle$. Setting $a'_1 := a_1 + a_i$, we can see that conditions (1)--(3) above hold for $n=1$. Now suppose that we have already found $a'_1, \dots, a'_n \in M$ such that conditions (1)--(3) hold. Fix $m \in \nn$ such that $s'_n$ divides $s_m$ in $M$. Set $Q := q_0 - \langle a'_1, \dots, a'_n \rangle$. From the fact that $q_0 \sim a'_i$ for every $i \in \ldb 1,n \rdb$, one can infer that the set $G^+ \cap Q$ is finite. Also,  $q = \textbf{O}(a_{m+1})$ for each $q \in G^+ \cap Q$. Therefore the set $G^+ \cap \big\{ \frac{q}n - a_{m+1} \mid q \in Q \text{ and } n \in \nn \big\}$ is also finite. As $(a_n)_{n \in \nn}$ is strictly decreasing, there is a $j \in \nn$ with $j > m+1$ such that $a_j \notin \big\{ \frac qn - a_{m+1} \mid q \in Q \text{ and } n \in \nn \big\}$. Thus, $Q$ is disjoint from $\langle a_{m+1} + a_j \rangle$. So after setting $a'_{n+1} = a_{m+1} + a_j$, we see that $q_0 \notin \langle a'_1, \dots, a'_{n+1} \rangle$ and also that $q_0 \sim a_{m+1} + a_j \sim a'_{n+1}$. In addition, observe that $s'_{n+1} := \sum_{i=1}^{n+1} a'_i = s'_n + a_{m+1} + a_j$ divides $s_m + a_{m+1} + a_j$ in $M$, which implies that $s'_{n+1}$ divides $s_j$ in $M$. Hence we can assume the existence of a sequence $(a'_n)_{n \in \nn}$ satisfying conditions (1)--(3) above.
	
	Finally, let $(r_n)_{n \in \nn}$ be the sequence defined as follows: take $r_0 = q_0$ and take $r_n = q_0 - s'_n$ for every $n \in \nn$. Because $r_n = q_0 - s'_n = q_0 - s'_{n+1} + (s'_{n+1} - s'_n) = r_{n+1} + a'_{n+1}$, we see that $(r_n + M)_{n \in \nn}$ is a non-stabilizing ascending chain of principal ideals of~$M$. Now consider the submonoid $N' := \langle a'_n, r_n \mid n \in \nn \rangle$ of $M$. Since $M$ is hereditarily atomic, $N'$ is atomic. In addition, the fact that $(r_n + M)_{n \in \nn}$ is a non-stabilizing ascending chain of principal ideals implies that $\mathcal{A}(N') \subseteq \{a'_n \mid n \in \nn\}$. This, along with the fact that $q_0 = r_1 + a'_1 \in N'$, guarantees the existence of $c_1, \dots, c_n \in \nn_0$ such that $q_0 = \sum_{i=1}^n c_i a'_i$. However, this contradicts that $q_0 \notin \langle a'_1, \dots, a'_n \rangle$.
	\smallskip
	
	\textsc{Case 2:} Each Archimedean class of $G$ contains only finitely many terms of the sequence $(a_n)_{n \in \nn}$. Then there exists a subsequence $(b_n)_{n \in \nn}$ of $(a_n)_{n \in \nn}$ such that each Archimedean class of $G$ contains at most one term of $(b_n)_{n \in \nn}$ and, after replacing $(q_n)_{n \in \nn}$ by $(q_0 - \sum_{i=1}^n b_i)_{n \in \nn}$, one can assume that each Archimedean class of $F$ contains at most one term of the sequence $(a_n)_{n \in \nn}$.
		
	Because $q_0 > \sum_{i=1}^n a_i$ for every $n \in \nn$, we see that $a_n = \textbf{O}(q_0)$ for every $n \in \nn$. On the other hand, the fact that $(q_n + M)_{n \in \nn}$ is a non-stabilizing ascending chain of principal ideals, together with the fact that $\langle a_n, q_n \mid n \in \nn \rangle$ is an atomic monoid, guarantees that $q_0 \in \langle a_n \mid n \in \nn \rangle$, which in turns implies that $q_0 \sim a_{k_1}$ for some $k_1 \in \nn$. Observe that $a_{k_1} \gg a_n$ for any $n > k_1$. Now suppose that we have found $k_1, \dots, k_m \in \nn$ with $k_1 < \dots < k_m$ such that $q_0 \sim a_{k_1} \gg \dots \gg a_{k_m} \gg a_n$ for all $n > k_m$. Then set $q'_0 := q_0 - \sum_{i=1}^m a_{k_i}$ and, for each $n \in \nn$, set $a'_n := \sum_{i=1}^n a_{k_m + i}$ and $q'_n = q'_0 - a'_n$. Observe that $(q'_n + M)_{n \in \nn}$ is a non-stabilizing ascending chain of principal ideals of~$M$. Proceeding as we did before, we find that $q'_0 \sim a_{k_m + j}$ for some $j \in \nn$. Since $q'_0 > \sum_{i=1}^n a_{k_m + i}$ for all $n \in \nn$, it follows that $a_{k_m + j} \gg a_n$ for every $n > k_m + j$. Then after setting $k_{m+1} := k_m + j$, we obtain $q_0 \sim a_{k_1} \gg \dots \gg a_{k_{m+1}} \gg a_n$ for all $n > k_{m+1}$. Hence, after replacing $(q_n)_{n \in \nn}$ by $(q_0 - \sum_{i=1}^n a_{k_i})_{n \in \nn}$, one can assume that $q_0 \sim a_1$ and also that $a_n \gg a_{n+1}$ for every $n \in \nn$. Now fix any $\ell \in \nn$. Since $q_\ell = q_{\ell +1} + a_{\ell +1}$, we see that $a_{\ell +1} = \textbf{O}(q_\ell)$. On the other hand, the atomicity of $\langle a_n, q_n \mid n \ge \ell +1 \rangle$ ensures that $q_\ell \in \langle a_n \mid n \ge \ell +1 \rangle$, and so $a_{\ell +1} \gg a_{\ell +n}$ for every $n \ge 2$ guarantees that $q_\ell \sim a_{\ell +1}$.
	
	Finally, write $q_0 = \sum_{i=1}^k c_i a_i$ for some $c_1, \dots, c_k \in \nn_0$. Note that $c_1 \ge 1$ because $q_0 \sim a_1 \gg a_n$ for every $n \ge 2$. Since $a_1 \gg a_2$, the fact that $(c_1 - 1)a_1 + \sum_{i=2}^k c_i a_i = q_0 - a_1 = q_1 \sim a_2$ implies that $c_1 = 1$. Set $j := \max\{ i \in \ldb 1, k \rdb \mid c_i = 1 \}$, and observe that $j < k$ as $q_0 - \sum_{i = 1}^k a_i = q_k \neq 0$. As $a_{j+1} \gg a_n$ for every $n \ge j+2$, the fact that $\sum_{i=j+1}^k c_i a_i = q_0 - \sum_{i=1}^j a_i = q_j \sim a_{j+1}$ guarantees the inequality $c_{j+1} \ge 1$. Finally, since $a_{j+1} \gg a_{j+2}$, the fact that $(c_{j+1} - 1)a_{j+1} + \sum_{i=j+2}^k c_i a_i = q_0 - \sum_{i=1}^{j+1} a_i = q_{j+1} \sim a_{j+2}$ implies that $c_{j+1} = 1$. However, this contradicts the maximality of $j$.
\end{proof}

With notation as in Theorem~\ref{thm:ACCP=HA}, the condition that the monoid is reduced is not superfluous. This is illustrated in the following example.

\begin{example}
	The additive abelian group $\zz^2$ trivially satisfies the ACCP as a monoid. To see that $\zz^2$ is not hereditarily atomic, let $M$ be the nonnegative cone of $\zz^2$ with respect to the lexicographical order with priority in the second coordinate. It is clear that $M$ is a submonoid of $\zz^2$, and it is not hard to verify that $\mathcal{A}(M) = \{(1,0)\}$. Thus, $M$ is not atomic, and so $\zz^2$ is not a hereditarily atomic monoid.
\end{example}

On the other hand, the condition of being reduced is not needed for the other implication of Theorem~\ref{thm:ACCP=HA}, as the following proposition shows.

\begin{prop}\footnote{Proposition~\ref{prop:HA implies ACCP} was kindly observed by Ben Li, who also suggested the given proof.} \label{prop:HA implies ACCP}
	If a torsion-free monoid is hereditarily atomic, then it satisfies the ACCP.
\end{prop}

\begin{proof}
	 Let $M$ be a hereditarily atomic torsion-free monoid and assume, towards a contradiction, that $M$ does not satisfy the ACCP. 
	 Let $(b_n + M)_{n \in \nn}$ be an ascending chain of principal ideals of $M$ that does not stabilize. Set $c_n := b_n - b_{n+1}$ for every $n \in \nn$. After replacing $(b_n + M)_{n \in \nn}$ by a suitable subsequence, we can assume that $c_n \notin \uu(M)$ for any $n \in \nn$. 
	 Now consider the submonoid $M' := \langle b_n, c_n \mid n \in \nn \rangle$ of $M$. Since $M$ is a hereditarily atomic torsion-free monoid, so is~$M'$. Observe that $(b_n + M')_{n \in \nn}$ is an ascending chain of principal ideals of $M'$, which does not stabilize in $M'$ because $(b_n + M)_{n \in \nn}$ does not stabilize in $M$. Thus, $M'$ does not satisfy the ACCP and, therefore $M'$ is not reduced by virtue of Theorem~\ref{thm:ACCP=HA}. As a result, some element in the defining generating set of $M'$ must be invertible. Since $c_n \notin \uu(M')$ for any $n \in \nn$, there exists $m \in \nn$ such that $b_m \in \uu(M') \subseteq \uu(M)$. This implies that $b_n + M = M$ for every $n \ge m$, which contradicts that the chain of principal ideals $(b_n + M)_{n \in \nn}$ does not stabilize in $M$.
\end{proof}

The question of whether the torsion-free condition is superfluous in Proposition~\ref{prop:HA implies ACCP} is still unanswered.

\begin{question}
	Does every hereditarily atomic monoid satisfy the ACCP?
\end{question}

\bigskip
\section{Atomicity in Positive Monoids}
\label{sec:positive monoids}

In this section, we mostly consider totally ordered monoids whose sets of nonzero elements are not in a neighborhood of $0$. In terms of atomicity, this condition is quite strong for a positive monoid of the additive group of an Archimedean field: indeed, if $M$ is such a monoid and $0$ is not a limit point of $M^\bullet$ in the order topology, then $M$ is a BFM \cite[Proposition~4.5]{fG19}. The same result can be easily extended to any positive monoid.

\begin{prop} \label{prop:M is BF when 0 is not a limit point for Archimedean groups}
	Let $M$ be a positive monoid of an Archimedean ordered group $G$. If $0$ is not a limit point of $M^\bullet$ in $G$, then $M$ is a BFM.
\end{prop}

\begin{proof}
	By virtue of H\"older's theorem, we can identify $G$ with an additive subgroup of $\rr$, in which case,~$M$ is an additive submonoid of $\rr$. Hence it follows from \cite[Proposition~4.5]{fG19} that $M$ is a BFM.
\end{proof}

When $M$ is a positive monoid of a totally ordered group that is not necessarily Archimedean, the fact that~$0$ is not a limit point of $M^\bullet$ is not a strong condition from the factorization-theoretical point of view. The following examples shed some light upon this observation.

\begin{example}
	Consider the abelian group $G := \qq \times \qq$ with the lexicographical order with priority in the first component; that is, $(x,y) \le (x', y')$ in $G$ provided that either $x < x'$ or $x=x'$ and $y \le y'$. Observe that $G^+ = (\{0\} \times \qq_{\ge 0}) \cup (\qq_{> 0} \times \qq)$.
	\begin{enumerate}
		\item \emph{A positive monoid of $G$ that is antimatter.} The monoid $M := \{(0,0)\} \cup (\qq_{> 0} \times \qq)$ of $G$ is a positive monoid of $G$, and $(0,0)$ is not a limit point of $M^\bullet$ with respect to the order topology of $G$. Also, $M$ satisfies that for each $b \in M^\bullet$, the element $\frac 12 b$ also belongs to $M^\bullet$. As a consequence,~$M$ is antimatter.
		\smallskip
		
		\item \emph{A positive monoid of $G$ that is neither antimatter nor atomic.} Now consider the positive monoid $M := G^+ \cap (\zz \times \zz) =  (\{0\} \times \nn_0) \cup (\nn \times \zz)$ of $G$. As in the previous example, $(0,0)$ is not a limit point of $M^\bullet$ with respect to the order topology of~$G$. Since $a := (0,1)$ is the minimum nonzero element of $M$, it must be an atom. On the other hand, each nonzero element of $M$ is divisible by $a$ and, therefore, $\mathcal{A}(M) = \{a\}$. However, $\langle a \rangle$ is a proper submonoid of $M$, which implies that $M$ is not atomic.
	\end{enumerate}
\end{example}

We proceed to construct an atomic positive monoid that is not strongly atomic.
		
\begin{example}  \label{ex:ATM not SAM}
	Fix $q \in \qq$ such that $q \in (0,1)$ and $q^{-1} \notin \nn$, and then set $M_q := \langle q^n \mid n \in \nn_0 \rangle$. It follows from \cite[Theorem~6.2]{GG18} that $M_q$ is an atomic monoid with $\mathcal{A}(M_q) = \{q^n \mid n \in \nn_0 \}$. Take distinct irrational numbers $\alpha, \beta \in \rr_{> 0}$ with $1 < \alpha < \beta$ such that the set $\{1,\alpha,\beta\}$ is linearly independent over~$\qq$. Now set~$S := \{s \in M_q \mid s < \alpha\}$. As $1 < \alpha$, the inclusion $\mathcal{A}(M_q) \subseteq S$ holds. Observe that~$S$ is a countable set. Let $\varphi \colon S \to \pp$ be an injective function. Now consider the monoid
	\[
		M_{\alpha, \beta} = \bigg\langle s, \frac{\alpha - s}{\varphi(s)}, \frac{\beta - s}{\varphi(s)} \ \bigg{|} \ s \in S \bigg\rangle.
	\]
	Because $\sup S = \alpha$ and $\alpha < \beta$, we see that $M_{\alpha, \beta}$ is a positive monoid of $\rr$. Note that $\alpha, \beta \in M_{\alpha, \beta}$ and also that $M_q$ is a submonoid of $M_{\alpha, \beta}$. Using the fact that the set $\{1,\alpha, \beta\}$ is linearly independent over~$\qq$, we can verify that no element of $S$ is divisible by neither $\frac{\alpha-s}{\varphi(s)}$ nor $\frac{\beta-s}{\varphi(s)}$ in~$M_{\alpha, \beta}$ for any $s \in S$. From this, we can deduce that every atom of $M_q$ is also an atom of $M_{\alpha, \beta}$; that is, $\{q^n \mid n \in \nn_0\} \subseteq \mathcal{A}(M_{\alpha, \beta})$. Therefore $\mathcal{A}(M_{\alpha, \beta}) \cap S = \{q^n \mid n \in \nn_0\}$. Once again we can use that $\{1,\alpha, \beta\}$ is linearly independent over $\qq$ to argue that $\frac{\alpha-s}{\varphi(s)}$ and $\frac{\beta-s}{\varphi(s)}$ are atoms of $M_{\alpha, \beta}$ for every $s \in S$. As a result,
	\[
		\mathcal{A}(M_{\alpha, \beta}) = \big\{ q^n \mid n \in \nn_0 \big\} \bigcup \bigg\{ \frac{\alpha-s}{\varphi(s)}, \frac{\beta-s}{\varphi(s)} \ \bigg{|} \ s \in S \bigg\}.
	\]
	It is clear now that every element of $M_{\alpha, \beta}$ can be expressed as a sum of atoms, which means that~$M_{\alpha, \beta}$ is an atomic monoid.
	
	Let us proceed to argue that~$M_{\alpha, \beta}$ is not strongly atomic. To do so, suppose that $d \in M$ is a common divisor of $\alpha$ and $\beta$ in~$M_{\alpha, \beta}$. It follows from the linearly independence of $\{1,\alpha, \beta\}$ that $\frac{\beta-s}{\varphi(s)} \nmid_{M_{\alpha,\beta}} \alpha$ for any $s \in S$. Similarly, $\frac{\alpha-s}{\varphi(s)} \nmid_{M_{\alpha,\beta}} \beta$ for any $s \in S$. Therefore $d \in M_q$, and so the inequality $d < \alpha$ ensures that $d \in S$. Now, after taking $k \in \nn$ sufficiently large so that $d + q^k \in S$, we see that
	\[
		\alpha-d = \varphi(d+q^k) \frac{\alpha-(d+q^k)}{\varphi(d+q^k)} + q^k \quad \text{ and } \quad \beta-d = \varphi(d+q^k) \frac{\beta-(d+q^k)}{\varphi(d+q^k)} + q^k,
	\]
	which implies that $q^k$ is a nonzero common divisor of both $\alpha-d$ and $\beta-d$ in $M_{\alpha, \beta}$. Therefore there is no common divisor $d$ of $\alpha$ and $\beta$ in $M_{\alpha, \beta}$ such that the only common divisor of $\alpha-d$ and $\beta-d$ is~$0$. Hence $M_{\alpha, \beta}$ is not strongly atomic.
		
	Now we consider the monoid $M := M_{\alpha, \beta} \times \nn_0$. Observe that $M$ is a positive monoid of the abelian group $G := \rr \times \zz$ endowed with the lexicographical order with priority in the first component. Also, we note that $(0,0)$ is not a limit point of $M^\bullet$ with respect to the order topology of~$G$. It is routine to check that the direct product of two atomic monoids is also an atomic monoid. As a consequence, $M$ is atomic. 
	Since~$M$ has a divisor-closed submonoid that is isomorphic to $M_{\alpha, \beta}$, namely $M_{\alpha, \beta} \times \{0\}$, we conclude that $M$ is not strongly atomic.
\end{example}

Now we exhibit a positive monoid that is strongly atomic but does not satisfy the ACCP.

\begin{example} \label{ex:SAM not ACCP}
	As in the previous example, for $q \in \qq \cap (0,1)$ with $q^{-1} \notin \nn$, consider the Puiseux monoid $M_q := \langle q^n \mid n \in \nn_0 \rangle$ whose set of atoms is $\mathcal{A}(M_q) = \{q^n \mid n \in \nn_0 \}$. Indeed, it follows from \cite[Example~3.8]{GL23} and \cite[Proposition~3.10(2)]{GL23} that $M_q$ is a strongly atomic monoid. Now observe that $\mathsf{d}(q) q^n = (\mathsf{d}(q) - \mathsf{n}(q)) q^n + \mathsf{d}(q) q^{n+1}$ for every $n \in \nn$, which implies that $\big( \mathsf{d}(q) q^n + M_q \big)_{n \ge 0}$ is an ascending chain of principal ideals of $M_q$ that does not stabilize. As a result, $M_q$ does not satisfy the ACCP. Following the lines of the last paragraph of Example~\ref{ex:ATM not SAM}, we can verify that $M := M_q \times \nn_0$ is a positive monoid of the totally ordered abelian group $\qq \times \zz$ (under the lexicographical order with priority in the first component) such that $M$ is strongly atomic, $M$ does not satisfy the ACCP, and $(0,0)$ is not a limit point of $M^\bullet$ with respect to the order topology of $\qq \times \zz$.
\end{example}

Lastly, we provide an example of a positive monoid that satisfies the ACCP but is not a BFM.
		
\begin{example} \label{ex:ACCP not BF}
 Consider the Puiseux monoid $M_0 := \big\langle \frac 1p \mid p \in \pp \big\rangle$ of $\qq$. It is known that~$M_0$ satisfies the ACCP (see \cite[Example~3.3]{AG22}). On the other hand, one can easily check that $\mathcal{A}(M_0) = \big\{\frac 1p \mid p \in \pp \big\}$. This implies that $\pp \subseteq \mathsf{L}_{M_0}(1)$ and, therefore, $M_0$ is not a BFM. Proceeding as in the previous two examples, we can see that $M := M_0 \times \nn_0$ is a positive monoid of the totally ordered group $\qq \times \zz$ (under the lexicographical order with priority in the first component) such that $(0,0)$ is not a limit point of $M^\bullet$ with respect to the order topology. Finally, from the fact that $M_0$ satisfies the ACCP but is not a BFM, we can deduce that $M$ satisfies the ACCP but is not a BFM.
\end{example}
\smallskip

We say that a positive monoid is \emph{increasing} if it can be generated by an increasing sequence. It is clear that if $M$ is an increasing monoid, then there is a neighborhood of $0$ disjoint from $M^\bullet$. Increasing positive monoids of the additive group of an ordered field are FFMs (see \cite[Theorem~5.6]{fG19}). The same statement holds for any increasing positive monoid.

\begin{theorem}
	Each increasing positive monoid of a totally ordered group is an FFM.
\end{theorem}

\begin{proof}
	The proof of \cite[Theorem~5.6]{fG19} can be mimicked as it does not use the multiplicative structure of the ordered field.
\end{proof}

For a totally ordered group $G$ and a nonzero $a \in G^+$, we name the positive monoid $M_a := \{0\} \cup G_{\ge a}$ following the terminology in~\cite{BG20}.

\begin{definition}
	Let $G$ be a totally ordered group $G$, and take a nonzero $a \in G^+$. Then we call $M_a$ the positive monoid of $G$ \emph{conducted} by $a$ or simply a \emph{conductive positive monoid}.
\end{definition}

For the rest of this section, we restrict our attention to conductive positive monoids. A version of conductive positive monoids was considered in \cite{BG20} and a special class of them was more recently considered in~\cite{BCG21}.

It immediately follows from Proposition~\ref{prop:M is BF when 0 is not a limit point for Archimedean groups} that when $G$ is Archimedean, $M_a$ is a BFM for each nonzero $a \in G^+$. However, when $G$ is not Archimedean, $M_a$ may not be even atomic for some nonzero $a \in G^+$. In the next proposition, for any nontrivial totally ordered group $G$, we provide conditions equivalent to that of $M_a$ being atomic. First, let us take a look at the following motivating example.

\begin{example} \label{ex:BF in Z^2}
	Consider the totally ordered group $G = \zz \times \zz$ with the lexicographical order with priority in the first component, in which case, $G^+ = (\{0\} \times \nn_0) \cup (\nn \times \zz)$. Observe that $G \setminus \{0\}$ consists of two Archimedean classes, whose intersections with $G^+$ are $C_1 := \{0\} \times \nn$ and $C_2 := \nn \times \zz$. Fix a nonzero $a \in G^+$. If $a \in C_1$, then $M_a$ is not atomic by virtue of Proposition~\ref{prop:conductive positive monoids BF} as $v(a) \succ \min \Gamma_G$: indeed, $\langle [a,2a) \rangle^\bullet \subset C_1$ and $M_a = \langle [a,2a) \rangle \cup C_2$. If $a \in C_2$, then $v(a) = \min \Gamma_G$ and so $M_a$ is a BFM by Proposition~\ref{prop:conductive positive monoids BF}: indeed, in this case, the restriction of the function $M_a \to \nn_0$ given by $(m,n) \mapsto m$ is clearly a length function of $M_a$ (see \cite[Proposition~3.1]{AG22}).
\end{example}

With Example~\ref{ex:BF in Z^2} in mind, we can characterize the conductive positive monoids that are atomic in several ways.

\begin{prop} \label{prop:conductive positive monoids BF}
	Let $G$ be a nontrivial totally ordered group, and fix a nonzero $a \in G^+$. Then $\mathcal{A}(M_a) = [a,2a)$, and the following conditions are equivalent.
		\begin{enumerate}
			\item[(a)] $M_a$ is a BFM.
			\smallskip
			
			\item[(b)] $M_a$ satisfies the ACCP.
			\smallskip
			
			\item[(c)] $M_a$ is strongly atomic.
			\smallskip
			
			\item[(d)] $M_a$ is atomic.
			\smallskip
			
			\item[(e)] $M_a$ is nearly atomic.
			\smallskip
			
			\item[(f)] $M_a$ is almost atomic.
			\smallskip
			
			\item[(g)] $M_a$ is quasi-atomic.
			\smallskip
			
			\item[(h)] $M_a^\bullet \subset v(a)$. 
			\smallskip
			
			\item[(i)] $v(a) = \min \Gamma_G$.
	\end{enumerate}
\end{prop}

\begin{proof}
	For each $x \in M_a$ with $x \ge 2a$, we see that $x-a \in M^\bullet_a$ and, therefore, $\mathcal{A}(M_a) \subseteq [a,2a)$. On the other hand, $2a$ is a lower bound for the set $G_{\ge a} + G_{\ge a}$, which implies that $\mathcal{A}(M_a) \subseteq [a,2a)$. Thus, $\mathcal{A}(M_a) = [a,2a)$. Let us show now that the conditions (a)--(i) above are equivalent.
	\smallskip
	
	(a) $\Rightarrow$ (b) $\Rightarrow$ (c) $\Rightarrow$ (d) $\Rightarrow$ (e): This is clear.
	\smallskip
	
	(e) $\Rightarrow$ (f): Suppose that $M$ is nearly atomic, and then take $b \in M$ such that $b+c$ is atomic for every nonzero $c \in M$. If $b=0$, then $M_a$ is atomic and, therefore, almost atomic. Suppose, on the other hand, that $b > 0$. In this case, $2b$ is an atomic element of $M$. In addition, for each $c \in M$, the element $2 b + c = b + (b + c)$ is an atomic element of $M$. As a consequence, $M$ is almost atomic.
	\smallskip
	
	(f) $\Rightarrow$ (g): This is clear.
	\smallskip
	
	(g) $\Rightarrow$ (h): Assume that $M_a$ is quasi-atomic. Take $b \in M_a^\bullet$. As $M_a$ is quasi-atomic, we can pick $c \in M_a$ such that $b+c = a_1 + \cdots + a_n$ for some $a_1, \dots, a_n \in [a,2a)$. Hence $b \le b+c = a_1 + \cdots + a_n < 2na$. This implies that $b = \textbf{O}(a)$, and it is clear that $a = \textbf{O}(b)$. Thus, $b \sim a$ for all nonzero $b \in M_a$, whence the inclusion $M_a^\bullet \subset v(a)$ holds.
	\smallskip
	
	(h) $\Rightarrow$ (i): Suppose that the inclusion $M_a^\bullet \subset v(a)$ holds. Fix a nonzero $g \in G$, and let us verify that $v(g) \succeq v(a)$. After replacing $g$ by $-g$ if necessary, we can assume that $g \in G^+$. If $g < a$, then $v(g) \succeq v(a)$ and, otherwise, $g \in M_a \subset v(a)$, which implies that $v(g) = v(a)$.
	\smallskip
	
	(i) $\Rightarrow$ (a): Assume that $v(a) \preceq v(g)$ for all nonzero $g \in G$. Fix $b \in M_a^\bullet$. The fact that $v(a) \preceq v(b)$ implies that $b = \textbf{O}(a)$ and, therefore, $b \le Na$ for some $N \in \nn$. Now suppose that $b = a_1 + \dots + a_\ell$ for some $a_1, \dots, a_\ell \in M_a^\bullet$. Then $\ell a \le  a_1 + \cdots + a_\ell = b$ and so $\ell a \le b \le N a$. Thus, $b$ cannot be written as a sum of more than $N$ elements of $M_a^\bullet$. Now assuming that $\ell$ is as large as it could possibly be, we obtain that the summands in the right-hand side of $b = a_1 + \dots + a_\ell$ are atoms of $M_a^\bullet$, which implies not only that $b$ is an atomic element of $M_a$ but also that $\mathsf{L}_{M_a}(b)$ is bounded. Hence $M_a$ is a BFM.
\end{proof}

No two of the conditions (a)--(d) in Proposition~\ref{prop:conductive positive monoids BF} are equivalent in the class of positive monoids. Indeed, we have seen that the positive monoid $M_0$ of Example~\ref{ex:ACCP not BF} satisfies the ACCP but is not a BFM, the positive monoid $M_q$ of Example~\ref{ex:SAM not ACCP} is strongly atomic but does not satisfy the ACCP, and the positive monoid $M_{\alpha, \beta}$ of Example~\ref{ex:ATM not SAM} is atomic but not strongly atomic. Now we proceed to provide examples to illustrate that, in the class of positive monoids, no two of the conditions (d)--(g) are equivalent. Using the idea in Example~\ref{ex:ATM not SAM}, we begin with an example of a positive monoid that is nearly atomic but not atomic.

\begin{example}
	Let $\alpha$ be a positive irrational number, and let $\varphi \colon \qq_{\ge 0} \to \pp$ be an injective function. Then consider the positive monoid $M$ of the totally ordered abelian group $\rr$ defined as follows:
	\[
		M := \Big\langle q, \frac{\alpha + q}{\varphi(q)} \ \Big{|} \ q \in \qq_{\ge 0} \Big\rangle. 
	\]
	From the irrationality of $\alpha$, one can readily argue that $\frac{\alpha + q}{\varphi(q)} \in \mathcal{A}(M)$ for all $q \in \qq_{\ge 0}$. We proceed to show that $M$ is nearly atomic. In order to do so, we first observe that $\alpha = \varphi(0) \frac{\alpha}{\varphi(0)} \in M$, and then we claim that $\alpha + r$ is an atomic element of $M$ for each $r \in M$. Take $r \in M$, and observe that we can write $r = q_0 + a_1 + \dots + a_n$ for some $q_0 \in \qq_{\ge 0}$ and $a_1, \dots, a_n \in \mathcal{A}(M)$. Then
	\[
		\alpha + r = (\alpha + q_0) + \sum_{k=1}^n a_k = \varphi(q_0) \frac{\alpha + q_0}{\varphi(q_0)} + \sum_{k=1}^n a_k,
	\]
	which illustrates that $\alpha + r$ is an atomic element of $M$. As a consequence, $M$ is nearly atomic. Finally, the irrationality of $\alpha$ guarantees that no element in $M \cap \qq_{> 0}$ can be written as a sum of atoms in $M$, whence $M$ is not atomic.
\end{example}

Let us construct now an almost atomic positive monoid that is not nearly atomic.

\begin{example} \label{ex:AA not atomic}
	Let $M_0$ be the monoid generated by the reciprocals of the prime numbers; that is, $M_0 = \big\langle \frac 1p \mid p \in \pp \big\rangle$, and then let $G$ be the
	difference group of $M_0$. Now consider the Puiseux monoid $M := M_0 \cup G_{\ge 1}$. We have mentioned in Example~\ref{ex:ACCP not BF} that $\mathcal{A}(M_0) = \big\{ \frac 1p \mid p \in \pp \big\}$. Thus, the fact that $\max \mathcal{A}(M_0) < 1$ ensures that no element of $G_{\ge 1}$ divides any atom of $M_0$ in $M$. As a result, $\{ \frac 1p \mid p \in \pp\} \subseteq \mathcal{A}(M)$. Moreover, because each element in $M_{>1}$ is divisible by an atom $\frac 1p$ in $M$ for a sufficiently large $p \in \pp$, it follows that $\mathcal{A}(M) = \{ \frac 1p \mid p \in \pp\}$. Since any element of $M$ can be written as a difference between two elements of $M_0$, any element of $M$ can be written as a difference between two sums of atoms of $M$. Thus, $M$ is almost atomic.

	On the other hand, we claim that $M$ is not nearly atomic. Suppose, for the sake of a contradiction, that there exists $q \in M$ such that $q + r$ is an atomic element in $M$ for all $r \in M^\bullet$. Let $P_q$ be the set of all primes dividing $\mathsf{d}(q)$. Since $P_q$ is a finite set, and the series $\sum_{n \in \nn} \frac 1{p_n}$ is divergent (here $p_n$ is the $n$-th prime number), we can find a nonempty finite subset $S$ of $\pp$ such that $S \cap P_q$ is empty and $\sum_{p \in S} \frac 1p > q + 2$. Now consider the element $r : = 1 + \prod_{p \in S} \frac 1p$. Observe that $r \in M^\bullet$ and $r < 2$. Because $q + r$ is atomic in $M$, there exist $p'_1, \dots, p'_k \in \pp$ not necessarily distinct such that $q+r = \sum_{j=1}^k\frac 1{p'_j}$. Now the fact that $S \cap P_q$ is empty implies that $S \subseteq \{p'_1, \dots, p'_k\}$, and so
	\[
		q+r = \sum_{j=1}^k\frac 1{p'_j} \ge \sum_{p \in S} \frac 1p > q+2,
	\]
	which contradicts that $r < 2$. As a consequence, we conclude that $M$ is an almost atomic monoid that is not nearly atomic.
\end{example}

Lastly, we construct a quasi-atomic positive monoid that is not almost atomic.

\begin{example} \label{ex:quasi-atomic not almost atomic}
	Let $M$ be the additive submonoid of $\qq_{\ge 0}$ generated by the set $S := \zz[\frac 12]_{\ge 0} \cup \zz[\frac 13]_{\ge 4/3}$, which is a positive monoid of $\qq$. It is clear that $\frac 43 \in \mathcal{A}(M)$. Hence each element in $4\nn$ is an atomic element in $M$. Now for each $q \in M^\bullet$, we can take the element $b := (4\mathsf{d}(q)-1)q$, and conclude that $b + q = 4\mathsf{n}(q)$ is an atomic element in $M$. Hence~$M$ is quasi-atomic.
	
	To argue that $M$ is not almost atomic, we first verify that $\mathsf{d}(q)$ is a power of $3$ for every $q \in \mathcal{A}(M)$. To do so, let $q$ be an atom of $M$. Observe that $6 \nmid \mathsf{d}(q)$ as, if this were not the case, an element in the defining generating set $S$ of $M$ would have its denominator divisible by~$6$. In addition, observe that $\mathsf{d}(q)$ cannot be a power of $2$ as, otherwise, $\frac 12 q$ would divide $q$ in~$M$. Hence the denominator of each atom of $M$ is a power of $3$. As a consequence, $\mathcal{A}(M) \subseteq \zz[\frac 13]_{\ge 4/3}$, which implies that the difference group $G$ of $\langle \mathcal{A}(M) \rangle$ is contained in $\zz[\frac 13]$. Therefore the fact that $\frac 12 \in M \setminus G$ guarantees that $M$ is not almost atomic.
\end{example}

We proceed to characterize the conductive positive monoids that are FFMs.

\begin{prop} \label{prop:conductive positive monoids FF}
	Let $G$ be a nontrivial totally ordered group, and fix a nonzero $a \in G^+$. Then the following conditions are equivalent.
	\begin{enumerate}
		\item[(a)] $M_a$ is an FFM. 
		\smallskip
		
		\item[(b)] $G$ is a cyclic group.
	\end{enumerate}
\end{prop}

\begin{proof}
	(a) $\Rightarrow$ (b): 	Suppose first that $M_a$ is an FFM. Then $M_a$ is atomic, and so it follows from Proposition~\ref{prop:conductive positive monoids BF} that $\mathcal{A}(M_a) = [a,2a)$ and $v(a) = \min \Gamma_G$. Now observe that for each $b \in [a,2a)$, the element $3a - b$ belongs to $[a, 2a)$ and, therefore, $b + (3a - b)$ is a factorization of $3a$ in~$M_a$. As a consequence, the fact that $|\mathsf{Z}_{M_a}(3a)| < \infty$ guarantees that $|[a,2a)| < \infty$. Now observe that if there existed $g \in G \setminus \{0\}$ such that $v(a) \prec v(g)$, then $a + (G^+ \cap v(g)) \subseteq (a,2a)$, which is not possible because $(a,2a)$ is finite and $G^+ \cap v(g)$ contains the infinite set $\nn |g|$. Hence $|\Gamma_G| = 1$, which means that $G$ is Archimedean. By virtue of H\"older's theorem, we can assume that $G$ is a nontrivial additive subgroup of $\rr$. For the second part of the proof, we need the following claim.
	\smallskip
	
	\noindent \textit{Claim:} $G$ is dense in $\rr$ unless $G$ is cyclic.
	\smallskip
	
	\noindent \textit{Proof of Claim:} Suppose that $\text{rank} \, G \ge 2$, and take integrally independent  elements $g,h \in G$; that is, the elements $g$ and~$h$ are linearly independent over $\qq$. Then $h \neq 0$ and $\frac gh \notin \qq$. It is well known and not hard to verify that the set $\{m + n \frac gh \mid m,n \in \zz\}$ is dense in $\rr$. Hence $G$ is dense in $\rr$. On the other hand, suppose that $\text{rank} \, G = 1$. Fix a nonzero $g_0 \in G$, and then set $G_0 := \frac 1{g_0} G$. It is clear that $G_0$ is a copy of $G$ inside $\qq$. It is well known that every subgroup of $\qq$ is the union of a (possibly infinite) ascending sequence of cyclic groups (see \cite[Corollary~2.8]{rG84}). Therefore $G_0$ is either an infinite cyclic group or $0$ is a limit point of $G_0 \setminus \{0\}$, which implies that $G_0$ is dense in $\qq$. As a consequence, $G$ is either an infinite cyclic group or a dense subset of $\rr$, and the claim follows.
	\smallskip
	
	Now assume, for the sake of a contradiction, that $G$ is not cyclic. Then it follows from the established claim that $G$ is dense in $\rr$. However, this contradicts that the interval $[a,2a)$ of $G$ is finite. Hence $G$ must be a cyclic subgroup of $\rr$.
	\smallskip
	
	(b) $\Rightarrow$ (a): Finally, suppose that $G \cong \zz$. Since $M_a$ is a submonoid of~$G^+$ and $G^+ \cong \nn_0$, it follows that $M_a$ is isomorphic to a numerical monoid and, therefore, it is finitely generated. Hence $M_a$ is an FFM.
\end{proof}

We can use Propositions~\ref{prop:conductive positive monoids BF} and~\ref{prop:conductive positive monoids FF} in tandem to construct examples of BFMs that are not FFMs. To illustrate this, we revisit Example~\ref{ex:BF in Z^2}.

\begin{example} \label{ex:BFM not FFM}
	For the totally ordered group $G = \zz \times \zz$ with the lexicographical order with priority in the first component, we have seen in Example~\ref{ex:BF in Z^2} that $G \setminus \{0\}$ consists of two Archimedean classes and that $M_a$ is a BFM provided that $a \in G^+$ and $v(a) = \min \Gamma_G$. In addition, in light of Proposition~\ref{prop:conductive positive monoids FF}, the fact that $|\Gamma_G| = 2$ ensures that $M_a$ is not an FFM.
\end{example}

 In the following proposition, we characterize the conductive positive monoids that are LFMs.

\begin{prop} \label{prop:conductive positive monoids LF}
	Let $G$ be a nontrivial totally ordered group, and fix a nonzero $a \in G^+$. Then the following conditions are equivalent.
	\begin{enumerate}
		\item[(a)] $M_a$ is an LFM. 
		\smallskip
		
		\item[(b)] $G = \zz b$ for some nonzero $b \in G^+$, and $a \in \{b, 2b\}$.
	\end{enumerate}
\end{prop}

\begin{proof}
	(a) $\Rightarrow$ (b): Suppose that $M_a$ is an LFM. Then $M_a$ is also an FFM and, therefore, the group $G$ is cyclic by virtue of Proposition~\ref{prop:conductive positive monoids FF}. Take a nonzero $b \in G$ such that $G = \zz b$. After replacing $b$ by $-b$ if necessary, we can assume that $b \in G^+$. Assume, by way of contradiction, that $a \ge 3b$. It follows from Proposition~\ref{prop:conductive positive monoids BF} that $\mathcal{A}(M_a) = \ldb a, 2a-b \rdb$, and the fact that $a \ge 3b$ guarantees that $|\mathcal{A}(M_a)| \ge 3$. Take $c_1, c_2, c_3 \in \nn_{\ge 3}$ such that $c_1 b, c_2b$, and $c_3 b$ are distinct atoms of $M_a$ and assume, without loss of generality, that $c_1 < c_2 < c_3$. Now take $m,n \in \nn$ such that $m(c_2 - c_1) = n(c_3 - c_2)$, and observe that $(m+n) c_2 b$ and $m c_1 b + n c_3 b$ yield two distinct factorizations of the same element of $M_a$ having the same length. This contradicts, however, that $M_a$ is an LFM. Hence $a \le 2b$ and, as $a$ is a nonzero element in the positive cone of $\zz b$, we can conclude that $a \in \{b, 2b\}$.
	\smallskip
	
	(b) $\Rightarrow$ (a): Suppose now that $G = \zz b$ for some nonzero $b \in G^+$, and $a \in \{b, 2b\}$. Since $G$ is the infinite cyclic group, $M_a$ is isomorphic to a submonoid of the additive monoid $\nn_0$. As $M_a$ is generated by at most two elements, it follows from \cite[Example~2.13]{CS11} that $M_a$ is an LFM.
\end{proof}

We conclude discussing the half-factorial property in the context of conductive positive monoids. Following Zaks~\cite{aZ80}, we say that an atomic monoid $M$ is \emph{half-factorial} (HFM) if any two factorizations of the same element have the same length. Half-factoriality has been systematically studied for almost four decades (see the recent paper~\cite{PS20} and references therein): a survey on the advances on length-factoriality until 2000 was given by Chapman and Coykendall in~\cite{CC00}. Observe that a monoid is a UFM if and only if it is both an HFM and an LFM. It turns out that for conductive positive monoids the condition of being a UFM and that of being an HFM are equivalent. In fact, we can characterize both properties as follows.

\begin{prop} \label{prop:conductive positive monoids UF/HF}
	Let $G$ be a nontrivial totally ordered group, and fix a nonzero $a \in G^+$. Then the following conditions are equivalent.
	\begin{enumerate}
		\item[(a)] $M_a$ is a UFM. 
		\smallskip
		
		\item[(b)] $M_a$ is an HFM.
		\smallskip
		
		\item[(c)] $G$ is cyclic, and $M_a = G^+$.
	\end{enumerate}
\end{prop}

\begin{proof}
	(a) $\Rightarrow$ (b): This follows from the definitions.
	\smallskip
	
	(b) $\Rightarrow$(c): Assume that $M_a$ is an HFM. We claim that the open interval $(a,2a)$ is empty. Suppose, by way of contradiction, that this is not the case, and take $b \in (a,2a)$. Since $3a - b \in (a,2a)$, it follows from Proposition~\ref{prop:conductive positive monoids BF} that $3a-b \in \mathcal{A}(M_a)$. Since $a, b \in \mathcal{A}(M_a)$ by the same proposition, the equality $a + a + a = b + (3a-b)$ yields two factorizations of the same element with different lengths, which contradicts that $M_a$ is an HFM. Thus, $(a,2a)$ is empty, as we claimed. Then it follows from Proposition~\ref{prop:conductive positive monoids BF} that $\mathcal{A}(M_a) = \{a\}$. Now if $c \in G^+ \cap [0,a)$, then the fact that $2a - c \in (a,2a]$ ensures that $c=0$. Therefore $a$ must be the smallest nonzero element of $G^+$, which implies that $M_a = G^+$. Moreover, since $M_a$ is atomic, every element of $M_a$ must be the sum of copies of~$a$, which is the only atom of $M_a$. As a consequence, $G^+ = M_a = \nn_0 a$, and so $G = \zz a$ is the cyclic group.
	\smallskip
	
	(c) $\Rightarrow$ (a): If $G$ is cyclic and $M_a = G^+$, then $M_a \cong \nn_0$, whence $M_a$ is a UFM.
\end{proof}

In the direction of Proposition~\ref{prop:conductive positive monoids UF/HF}, it is worth mentioning that the property of being a UFM and that of being an HFM are also equivalent in the class consisting of all rank-one positive monoids. Indeed, every positive monoid is torsion-free, and it follows from~\cite[Theorem~3.12]{GGT21} that every rank-one torsion-free monoid is either a group or a Puiseux monoid (up to isomorphism). In addition, it was proved in \cite[Proposition~4.3]{fG20} that a Puiseux monoid is an HFM if and only if it is a UFM. Therefore we obtain the following remark.

\begin{remark}
	A rank-one positive monoid is a UFM if and only if it is an HFM.
\end{remark}

For higher ranks, however, there are positive monoids that are HFM but not UFM. Indeed, in the following example we exhibit a rank-two positive monoid that is an HFM but is not even an FFM.

\begin{example}
	Consider the totally ordered group $G = \zz \times \zz$ with the lexicographical order with priority in the first component, in which case, $G^+ = (\{0\} \times \nn_0) \cup (\nn \times \zz)$. Now consider the monoid $M := \{(0,0)\} \cup (\nn \times \zz)$, which is a positive monoid of $G$. We first observe that $\mathcal{A}(M) = \{(1,n) \mid n \in \zz\}$ and, therefore, $M$ is atomic. In addition, we can see that each factorization of an element $(m,n) \in M^\bullet$ consists of precisely $m$ atoms (counting repetitions). Hence $M$ is an HFM. On the other hand, since $(2,0) = (1,-n) + (1,n)$ for every $n \in \nn$, the element $(2,0)$ has infinitely many factorizations in $M$ (indeed, one can similarly see that every nonzero element of $M$ that is not an atom has infinitely many factorizations). As a consequence, $M$ is not an FFM.
\end{example}

We conclude this paper considering the atomic diagram in Figure~\ref{fig:full atomicity diagram} restricted to the class of conductive positive monoids. The implications in the diagram shown in Figure~\ref{fig:atomicity conductive PM} summarize the main results we have established in this section.

\begin{figure}[h]
	\begin{tikzcd}[cramped]
		\textbf{[ UFM } \ \arrow[r, Leftrightarrow] & \ \textbf{ HFM ]} \ \arrow[r, Rightarrow] & \ \textbf{ LFM } \arrow[r, Rightarrow] &  \textbf{ FFM } \arrow[r, Rightarrow] \ & \ \textbf{[ BFM } \ \arrow[r, Leftrightarrow] & \ \textbf{ QAM ]}
	\end{tikzcd}
	\caption{Adaptation of the chain of implications shown in Figure~\ref{fig:full atomicity diagram} for the class of conductive positive monoids. None of the three implication arrows is reversible in this class. The atomic properties between being a BFM and being a quasi-atomic monoid in Proposition~\ref{prop:conductive positive monoids BF} have been omitted for simplicity.}
	\label{fig:atomicity conductive PM}
\end{figure}

\bigskip
\section*{Acknowledgments}

During the preparation of the present paper both authors were part of the PRIMES program at MIT, and they would like to thank the directors and organizers of the program for making this collaboration possible. The first author kindly acknowledges support from the NSF under the awards DMS-1903069 and DMS-2213323.

\bigskip


\begin{thebibliography}{20}

	\bibitem{AAZ90} D.~D. Anderson, D.~F. Anderson, and M. Zafrullah: \emph{Factorizations in integral domains}, J. Pure Appl. Algebra \textbf{69} (1990) 1--19.

	\bibitem{AG22} D. F. Anderson and F. Gotti: \emph{Bounded and finite factorization domains}. In: Rings, Monoids, and Module Theory (Eds. A. Badawi and J. Coykendall) pp. 7--57. Springer Proceedings in Mathematics \& Statistics, Vol. 382, Singapore, 2022.

	\bibitem{BCG21} N. R. Baeth, S. T. Chapman, and F. Gotti: \emph{Bi-atomic classes of positive semirings}, Semigroup Forum \textbf{103} (2021) 1--23.
	
	\bibitem{BG20} N. R. Baeth and F. Gotti: \emph{Factorizations in upper triangular matrices over information semialgebras}, J. Algebra \textbf{562} (2020) 466--496.

	\bibitem{BS15} N. R. Baeth and D. Smertnig: \emph{Factorization theory: From commutative to noncommutative settings}, J. Algebra \textbf{441} (2015), 475--551.

	\bibitem{BC19} J. G. Boynton and J. Coykendall: \emph{An example of an atomic pullback without the ACCP}, J. Pure Appl. Algebra \textbf{223} (2019) 619--625.

	\bibitem{BC15} J. G. Boynton and J. Coykendall: \emph{On the graph divisibility of an integral domain}, Canad. Math. Bull. \textbf{58} (2015) 449--458.

	\bibitem{BVZ22} A. Bu, J. Vulakh, and A. Zhao: \emph{Length-factoriality and pure irreducibility}. Preprint available on arXiv: https://arxiv.org/abs/2210.06638

	\bibitem{CC00} S.~T. Chapman and J. Coykendall: \emph{Half-factorial domains, a survey}. In: \emph{Non-Noetherian Commutative Ring Theory} (Eds. S.~T. Chapman and S. Glaz) pp. 97--115, Mathematics and Its Applications, vol. 520, Kluwer Academic Publishers, Springer, Boston 2000.
	
	\bibitem{CCGS21} S.~T. Chapman, J. Coykendall, F. Gotti, and W. W. Smith: \emph{Length-factoriality in commutative monoids and integral domains}, J. Algebra \textbf{578} (2021) 186--212.

	\bibitem{CGG20} S. T. Chapman, F. Gotti, and M. Gotti: \emph{Factorization invariants of Puiseux monoids generated by geometric sequences}, Comm. Algebra \textbf{48} (2020) 380--396.
	
	\bibitem{CGG21} S.~T. Chapman, F. Gotti, and M. Gotti: \emph{When is a Puiseux monoid atomic?}, Amer. Math. Monthly \textbf{128} (2021) 186--212.

	\bibitem{pC68} P.~M. Cohn: \emph{Bezout rings and and their subrings}, Proc. Cambridge Philos. Soc. \textbf{64} (1968) 251--264.


	\bibitem{CT22} L. Cossu and S. Tringali: \emph{Abstract factorization theorems with applications to idempotent factorizations}. Preprint available on arXiv: https://arxiv.org/abs/2108.12379

	\bibitem{CT23} L. Cossu and S. Tringali: \emph{Factorization under local finiteness conditions}. Preprint available on arXiv: https://arxiv.org/abs/2208.05869

	
	\bibitem{CG19} J. Coykendall and F. Gotti: \emph{On the atomicity of monoid algebras}, J. Algebra \textbf{539} (2019) 138--151.

	\bibitem{CGH23} J. Coykendall, F. Gotti, and R. E. Hasenauer: \emph{Hereditary atomicity in integral domains}, J. Pure Appl. Algebra \textbf{227} (2023) 107249.

	\bibitem{CS11} J. Coykendall and W. W. Smith: \emph{On unique factorization domains}, J. Algebra \textbf{332} (2011) 62--70.
	
	\bibitem{DW96} H.~G. Dales and W.~H. Woodin: \emph{Super-Real Fields: Totally Ordered Fields with Additional Structure}, Oxford University Press, New York, 1996.

	\bibitem{EJ21} R. A. C. Edmonds and J. R. Juett: \emph{Associates, irreducibility, and factorization length in monoid rings with zero divisors}, Comm. Algebra \textbf{49} (2021) 1836--1860.

	\bibitem{lF70} L. Fuchs: \emph{Infinite Abelian Groups I}, Academic Press, 1970.

	\bibitem{GGT21} A. Geroldinger, F. Gotti, and S. Tringali: \emph{On strongly primary monoids, with a focus on Puiseux monoids}, J. Algebra \textbf{567} (2021) 310--345.

	\bibitem{GH06} A.~Geroldinger and F.~Halter-Koch: \emph{Non-unique Factorizations: Algebraic, Combinatorial and Analytic Theory}, Pure and Applied Mathematics Vol. 278, Chapman \& Hall/CRC, Boca Raton, 2006.
	
	\bibitem{GZ21} A. Geroldinger and Q. Zhong: \emph{A characterization of length-factorial Krull monoids}, New York J. Math. \textbf{27} (2021) 1347--1374.
	
	\bibitem{GZ20} A. Geroldinger and Q. Zhong: \emph{Factorization theory in commutative monoids}, Semigroup Forum \textbf{100} (2020) 22--51.

	\bibitem{rG84} R. Gilmer: \emph{Commutative Semigroup Rings}, The University of Chicago Press, 1984.
	

	\bibitem{fG19} F.~Gotti: \emph{Increasing positive monoids of ordered fields are FF-monoids}, J. Algebra \textbf{518} (2019) 40--56.

	\bibitem{fG20} F. Gotti: \emph{Irreducibility and factorizations in monoid rings}. In: \emph{Numerical Semigroups} (Eds. V. Barucci, S. T. Chapman, M. D'Anna, and R. Fr\"oberg) pp. 129--139, Springer INdAM Series, vol. \textbf{40}, Cham 2020.

	\bibitem{GG18} F. Gotti and M. Gotti: \emph{Atomicity and boundedness of monotone Puiseux monoids}, Semigroup Forum \textbf{96} (2018) 536--552.

	\bibitem{GL22} F. Gotti and B. Li: \emph{Divisibility in rings of integer-valued polynomials}, New York J. Math \textbf{28} (2022) 117--139.
	
	
	\bibitem{GL23} F. Gotti and B. Li: \emph{Divisibility and a weak ascending chain condition on principal ideals}. Preprint available on arXiv: https://arxiv.org/abs/2212.06213
	
	\bibitem{aG74} A.~Grams: \emph{Atomic rings and the ascending chain condition for principal ideals}, Math. Proc. Cambridge Philos. Soc. \textbf{75} (1974) 321--329.

	\bibitem{fHK92} F. Halter-Koch: \emph{Finiteness theorems for factorizations}, Semigroup Forum \textbf{44} (1992) 112--117.

	\bibitem{JM22} J. R. Juett and A. M. Medina: \emph{Finite factorization properties in commutative monoid rings with zero divisors}, Comm. Algebra \textbf{50} (2022) 392--422.

	\bibitem{nLL19} N. Lebowitz-Lockard: \emph{On domains with properties weaker than atomicity}, Comm. Algebra \textbf{47} (2019) 1862--1868.

	\bibitem{PS20} A. Plagne and W. A. Schmid: \emph{On congruence half-factorial Krull monoids with cyclic class group}, Journal of Combinatorial Algebra \textbf{3} (2020) 331--400.
		
	\bibitem{mR93} M. Roitman: \emph{Polynomial extensions of atomic domains}, J. Pure Appl. Algebra \textbf{87} (1993) 187--199.

	\bibitem{aZ82} A. Zaks: \emph{Atomic rings without a.c.c. on principal ideals}, J. Algebra \textbf{74} (1982) 223--231.
	
	\bibitem{aZ80} A. Zaks: \emph{Half-factorial domains}, Israel J. of Math. \textbf{37} (1980) 281--302.

\end{thebibliography}
\end{document}